\documentclass[11pt,reqno]{amsart}
\usepackage{amsthm,amssymb,amsmath,euscript}
\usepackage[pagewise]{lineno}
\usepackage{enumitem}
\usepackage{hyperref}
\usepackage{fix-cm}
\usepackage{graphicx}

\usepackage[T5,T1]{fontenc}
\usepackage{mlmodern}

\newtheorem{theorem}{Theorem}
\newtheorem*{theorem*}{Theorem}
\newtheorem{lemma}{Lemma}
\newtheorem{corollary}{Corollary}
\newtheorem{proposition}{Proposition}

\theoremstyle{definition}

\newtheorem{definition}{\bf Definition}
\newtheorem*{definition*}{\bf Definition}

\newtheorem{remark}{\bf Remark}
\newtheorem*{remark*}{\bf Remark}
\newtheorem*{remarks}{\bf Remarks}

\newtheorem*{example*}{\bf Example}

\newcommand{\loc}{{\rm loc}}

\newcommand{\Real}{{\rm Re}}

\newcommand{\sprt}{{\rm sprt\,}}

\numberwithin{equation}{section}

\begin{document}

\fontsize{10.5pt}{4.5mm}\selectfont

\title{Intrinsic Hardy inequality for fractional Kolmogorov operator}

\author{Damir Kinzebulatov}

\address{Universit\'{e} Laval, D\'{e}partement de math\'{e}matiques et de statistique, Qu\'{e}bec, QC, Canada}

\email{damir.kinzebulatov@mat.ulaval.ca}

\thanks{The research of the author is supported by the NSERC (grant RGPIN-2024-04236)}

\keywords{Hardy inequality, singular drift, fractional Laplacian, diffusion equation}

\begin{abstract}
We obtain Hardy inequality for non-local diffusion operator with singular drift, in the case when the strength of attraction to the origin by the drift takes the critical value.
\end{abstract}

\subjclass[2020]{46E35 (primary), 60H10 (secondary)}

\maketitle

\section{Introduction}

The subject of this paper is the non-local diffusion operator with critical attracting drift,
\begin{equation}
\label{kolm}
\tag{$\star$}
\Lambda = (-\Delta)^{\frac{\alpha}{2}} +   b \cdot \nabla, \quad 1<\alpha < 2,
\end{equation}
that we consider on the $d$-dimensional torus $\Pi^d$, $d \geq 2$, obtained from the cube $Q=(-2,2]^d \subset \mathbb R^d$ by appropriately identifying its faces.
The drift $b:\Pi^d \rightarrow \mathbb R^d$ satisfies
\begin{equation}
\label{b_def}
\tag{$\star\star$}
b(x)|_{ x \in B_1(0) }= \frac{\nu_\star x}{|x|^{\alpha}} 
\end{equation}
and is smooth everywhere else.
Invoking the probabilistic interpretation of $\Lambda$, one can view the coupling constant $\nu_\star>0$ as measuring the strength with which the drift pushes the isotropic $\alpha$-stable process to the origin (this can be made precise, see (\textit{ii}) in ``Further discussion''). 
This coupling constant is taken to be
\begin{equation}
\label{nu_star}
\tag{$\star\star\star$}
\nu_\star:=2^{\alpha-1}\frac{\Gamma(\frac{\alpha}{2})\Gamma(\frac{d}{2})}{\Gamma(\frac{d-\alpha}{2})}.
\end{equation}
This is its critical value (we justify this below). It cannot be attained using the standard means such as the Stroock-Varopoulos inequality. The latter is a deep observation due to \cite{BJLP} put here in the context of \cite{KSS} where the authors demonstrated that there exists a theory of operator $\Lambda$ for $b(x)=\nu\frac{x}{|x|^\alpha}$ in the sub-critical regime $\nu<\nu_\star$, including two-sided heat kernel bounds and some regularity results for the parabolic equation in $L^p$ with $p$ that needs to be chosen depending on the proximity of $\nu$ to $\nu_\star$. It is however not known what theory of $\Lambda$ exists in the critical regime $\nu=\nu_\star$. The goal of the present paper is to clarify this. Our main instrument and main result  will be appropriate fractional Hardy inequality (Theorems \ref{thm1} and \ref{thm2}).

The parabolic equation $(\partial_t+\Lambda)v=0$ can be viewed as a model (``two-particle'') case of the following interacting particle system in the $dN$-dimensional space:
\begin{equation}
\label{part_syst}
(\partial_t+\sum_{i=1}^N (-\Delta)^{\frac{\alpha}{2}}_{x_i} + \frac{\nu}{N}\sum_{i=1}^N \sum_{j=1, j \neq i}^N \frac{x_i-x_j}{|x_i-x_j|^\alpha} \cdot \nabla_{x_i})v=0. 
\end{equation}
When $\alpha=2$ and $d=2$, this is the finite particle approximation of the celebrated Keller-Segel model of chemotaxis which arises as the mean-field limit of the adjoint of \eqref{part_syst} as $N \rightarrow \infty$. Both for the Keller-Segel model and the finite particle system \eqref{part_syst}, understanding what happens at the critical values of the attraction strength $\nu$ has been the subject of much research, see \cite{C,FJ,FT} and references therein. It has been demonstrated  that when $\nu$ crosses the critical threshold, the particles start to collide and the corresponding stochastic differential equation  ceases to have a weak solution; there is a blow-up.

In what follows, we are mostly concerned with the fractional case $1<\alpha<2$. (The case $\alpha=2$ is included as well, provided that $d \geq 3$, but is somewhat trivial in the context of the present paper, see remark after Theorem \ref{thm1}.)

Set
$$
\kappa_\beta:=\frac{2^\alpha \Gamma(\frac{\beta+\alpha}{2})\Gamma(\frac{d-\beta}{2})}{\Gamma(\frac{\beta}{2})\Gamma(\frac{d-\beta-\alpha}{2})}.
$$
Let us take a short detour into \cite{BJLP} and fractional Schr\"{o}dinger operators. 

\begin{enumerate}
\item[(a)]
The authors considered fractional Schr\"{o}dinger operator $$H=(-\Delta)^{\frac{\alpha}{2}}-\mu|x|^{-\alpha} \quad \text{(form-difference)}
$$ on $\mathbb R^d$ with coupling constant $0 \leq \mu \leq \kappa_{\frac{d-\alpha}{2}}$. Taking into account
\begin{equation}
\label{kappa}
\kappa_{\frac{d-\alpha}{2}}=\max_{1<p<\infty} \kappa_{\frac{d-\alpha}{p}},
\end{equation}
they proved, among other results, the following $L^p$ fractional Hardy inequality ($1<p<\infty$) for $H$ with $\mu \leq\kappa_{\frac{d-\alpha}{p}} $:
\begin{equation}
\label{H1}
\langle \kappa_{\frac{d-\alpha}{p}}|x|^{-\alpha}u,u^{p-1}\rangle \leq \langle (-\Delta)^{\frac{\alpha}{2}}u,u^{p-1}\rangle, 
\end{equation}
where $u$ vanishes at infinity sufficiently rapidly.
Here $\langle \cdot \rangle$ denotes integration, $\langle \cdot, \cdot\rangle$ the inner product in $L^2$ and $u^{p-1}:=u|u|^{p-2}$.
The inequality \eqref{H1} appears e.g.\,when one verifies condition $\langle Hu,u^{p-1}\rangle \geq 0$ in order to prove $L^p$ contractivity of the Schr\"{o}dinger semigroup:
$$
\|e^{-tH}f\|_p \leq \|f\|_p, \quad t>0.
$$
\end{enumerate}

\smallskip

From the physical perspective, $L^2$ is the natural space for  Schr\"{o}dinger and related operators \cite{JKS}.
One ventures into $L^p$ to obtain additional  
information about the semigroup $e^{-tH}$, but at expense of imposing more restrictive condition on the coupling constant, cf.\,\eqref{kappa}. In this regard, we give the following, only slightly informal definition.

\begin{definition}
We say that a space in which one considers a perturbation of the (fractional) Laplacian is intrinsic for this operator if it allows to handle the maximal value of the coupling constant. The Hardy inequality in this space will be called intrinsic.
\end{definition}

Thus, $L^2$ is an intrinsic space for $H$. An intrinsic Hardy inequality for $H$ is the usual  $L^2$ fractional Hardy inequality \eqref{ord_H}.

\begin{enumerate}
\item[(b)]
When $p \neq 2$, $\alpha<2$, Hardy inequality \eqref{H1}  has strictly better constant than the $L^p$ fractional Hardy inequality that was known previously, i.e.\,a consequence of 
\begin{equation}
\label{ord_H}
\kappa_{\frac{d-\alpha}{2}}\langle |x|^{-\alpha},v^2\rangle \leq \langle (-\Delta)^{\frac{\alpha}{2}}v,(-\Delta)^{\frac{\alpha}{4}}v\rangle,
\end{equation} (see \cite{KPS}, see also \cite{FLS}) and the celebrated Stroock-Varopoulos inequality
\begin{equation}
\label{SV}
\langle (-\Delta)^{\frac{\alpha}{4}}u^{\frac{p}{2}},(-\Delta)^{\frac{\alpha}{4}}u^{\frac{p}{2}}\rangle \leq \frac{p^2}{4(p-1)} \langle (-\Delta)^{\frac{\alpha}{4}}u,u^{p-1} \rangle.
\end{equation}
Namely, upon taking $v=u^{\frac{p}{2}}$, these two inequalities give \eqref{H1} with constant $\frac{4(p-1)}{p^2}\kappa_{\frac{d-\alpha}{2}}$ in the left-hand side.
Whenever $p \neq 2$, $\alpha<2$, one has
$$
\frac{4(p-1)}{p^2}\kappa_{\frac{d-\alpha}{2}}<\kappa_{\frac{d-\alpha}{p}}.
$$

\begin{figure}[ht]
\label{fig1}
\includegraphics[scale=0.5]{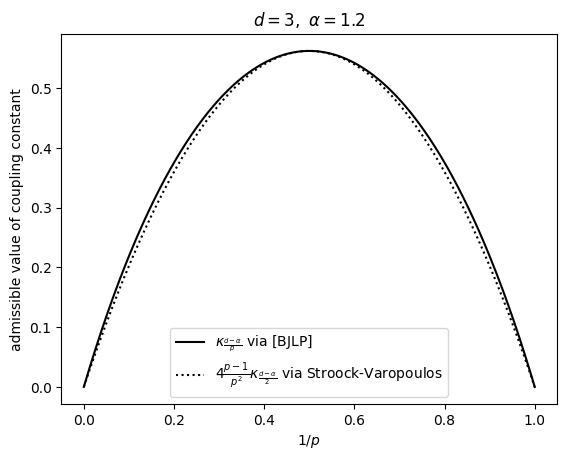}
\end{figure}

The authors of \cite{BJLP} show that the constant $\kappa_{\frac{d-\alpha}{p}}$ in their inequality \eqref{H1} is the best possible. 

\smallskip

\item[(c)]
The key element of their proof of \eqref{H1} is the following observation about the symmetries of $(-\Delta)^{\frac{\alpha}{2}}$: 
$$
\varphi_\beta(x)=|x|^{-\beta} \text{ is a Lyapunov function of $H$:} \quad (-\Delta)^{\frac{\alpha}{2}}\varphi_\beta=\kappa_\beta |x|^{-\alpha} \varphi_\beta
$$
and
$$
\varphi_\beta^p=\varphi_{p\beta}.
$$
The existence of a Lyapunov function is an abstract result: one can construct it from the heat kernel of the operator  \cite{BDK} (cf.\,\eqref{lyapunov_rep}). However, the fact that a power of the Lyapunov function is still a Lyapunov function, modulo adjusting the coupling constant, is specific to the fractional Laplacian. 

The Stroock-Varopoulos inequality \eqref{SV} is valid for all symmetric Markov generators and does not take these extra symmetries into account, see \cite{LS} for detailed discussion of \eqref{SV} and related inequalities.
\end{enumerate}

\medskip

We now return to the subject of this paper -- drift perturbation of the fractional Laplacian. Let us work temporarily on $\mathbb R^d$ and consider the strength of attraction to the origin $\nu$ strictly less than the critical value: 
$$\Lambda_\nu=(-\Delta)^{\frac{\alpha}{2}} + \nu \frac{x}{|x|^\alpha} \cdot \nabla, \quad 0 \leq \nu<\nu_\star.$$
The contractivity of the semigroup $e^{-t\Lambda_\nu}$ in $L^p$, i.e.
\begin{equation}
\label{contr_Lp}
\|e^{-t\Lambda_\nu}f\|_{p} \leq \|f\|_p
\end{equation}
can be established classically by combining \eqref{ord_H} and the Stroock-Varopoulos inequality \eqref{SV}, but at expense of imposing sub-optimal condition $\nu<\frac{p}{d-\alpha}4\frac{p^2}{p-1}\kappa_{\frac{d-\alpha}{2}}$.
In light of the result of \cite{BJLP}, the right way to proceed to prove \eqref{contr_Lp} is to use the next Hardy inequality. It is obtained from \eqref{H1} simply by integrating by parts in the drift term:
\begin{equation}
\label{H2}
\frac{p}{d-\alpha}\kappa_{\frac{d-\alpha}{p}}\langle - \frac{x}{|x|^\alpha} \cdot \nabla u,u^{p-1}\rangle \leq \langle (-\Delta)^{\frac{\alpha}{2}}u,u^{p-1}\rangle, \quad 1<p<\infty,
\end{equation}
where $u$ vanishes at infinity sufficiently rapidly. 
The dependence of the coupling constant $\nu =\frac{p}{d-\alpha}\kappa_{\frac{d-\alpha}{p}}$ on $p$ is as follows:

\begin{figure}[ht]
\label{fig2}
\includegraphics[scale=0.5]{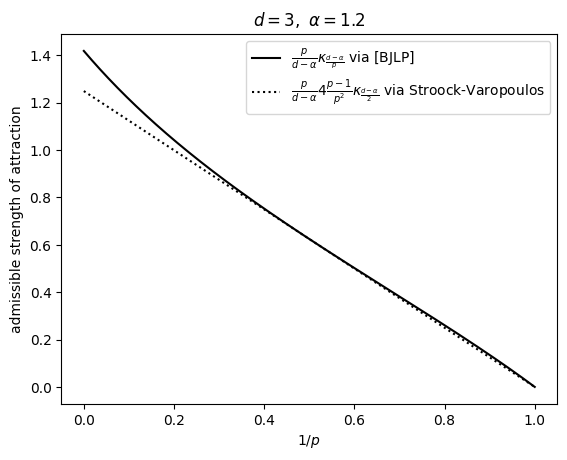}
\end{figure}

The maximal value of the coupling constant $\nu_\star$ is reached only in the limit:
\begin{equation}
\label{max_conv}
\frac{p}{d-\alpha}\kappa_{\frac{d-\alpha}{p}} \rightarrow \nu_\star \quad \text{ as } p \uparrow \infty.
\end{equation}
We are satisfied in this paper with the following two justifications for the maximality of $\nu_\star$, a formal one and an informal one. The formal one is that for every $p>1$ the constant in the Hardy inequality \eqref{H2} is the best possible -- this is proved in \cite{BJLP} (\eqref{H2} $\Leftrightarrow$ \eqref{H1}). The informal one is that the Lyapunov function (formal invariant density) for $\Lambda_\nu$ ceases to exist exactly as $\nu$ reaches $\nu_\star$, see remark 5 in the end of this introduction.

In order to approach $\nu_\star$  we need to work in $L^p$ with $p$ sufficiently large. This matches the fact that, from the probabilistic perspective, a natural space for diffusion operators\footnote{$C_\infty=\{f \in C(\mathbb R^d) \mid \lim_{x \rightarrow \infty}f(x)=0 \text{ with the sup-norm} \}$} is $C_\infty$ whose local topology is stronger than the one of $L^p$ regardless of how large $p$ is. In other words, $C_\infty$ is a good candidate for the intrinsic space of $\Lambda$. 
The problem is that by taking $p \rightarrow \infty$ in \eqref{contr_Lp}, \eqref{H2} directly, we arrive at a fundamental but also in some sense trivial a priori estimate
$$
\|e^{-t\Lambda_\nu}f\|_{\infty} \leq \|f\|_\infty, \quad t>0.
$$
Namely, there is no Hardy inequality in $L^\infty$ and it is not clear what could be the weak solution theory of the diffusion equation in this space. Further, regarding the $C_0$ semigroup theory of $\Lambda_\nu$ (which of course we cannot have in $L^\infty$), to reach $C_\infty$, it seems like we still need some kind of regularity theory of the Kolmogorov operator, which disappears as $p \uparrow \infty$, see remark (\textit{ii}) regarding De Giorgi's method in ``Further discussion''. So, arguably, by taking $p \rightarrow \infty$ we are losing information.
We can, however, pass to the limit $p \uparrow \infty$ indirectly and arrive to a non-trivial Hardy inequality for $\nu=\nu_\star$ and a weak solution theory. To this end, we formally take in \eqref{H2} $u=1+\frac{v}{p}$, arriving at
\begin{equation*}
\frac{p}{d-\alpha}\kappa_{\frac{d-\alpha}{p}}\left\langle - \frac{x}{|x|^\alpha} \cdot \nabla (1+\frac{v}{p}),(1+\frac{v}{p})^{p-1}\right\rangle \leq \langle (-\Delta)^{\frac{\alpha}{2}}(1+\frac{v}{p}),(1+\frac{v}{p})^{p-1}\rangle.
\end{equation*}
Both $\nabla$ and $(-\Delta)^{\frac{\alpha}{2}}$ annihilate constants, so multiplying the previous inequality by $p$, we obtain
\begin{equation*}
\frac{p}{d-\alpha}\kappa_{\frac{d-\alpha}{p}}\left\langle - \frac{x}{|x|^\alpha} \cdot \nabla v,(1+\frac{v}{p})^{p-1}\right\rangle \leq \langle (-\Delta)^{\frac{\alpha}{2}}v,(1+\frac{v}{p})^{p-1}\rangle.
\end{equation*}
Now we can take $p \uparrow \infty$, arriving at a non-trivial Hardy inequality -- this is our main observation in this paper -- but only after we address the formal nature of the above calculations. Indeed, $u=1+\frac{v}{p}$ is not an admissible function in \eqref{H2}. It is admissible if we work on the torus $\Pi^d$. Working on the torus requires adding in the last inequality a constant term with proper dependence on $p$.  One difficulty of working on the torus is that one no longer has a nice expression for the Lyapunov function: we need both appropriate behaviour of the Lyapunov function at the singularity and the scaling property, cf.\,(c) above. We will address this in Section \ref{prelim_sect}.

\subsection{Main results}
From now on, we work on the torus $\Pi^d$. We write
$$
\langle f \rangle:=\int_{\Pi^d} f(x) dx, \quad \langle f,g\rangle:=\langle f g\rangle
$$
(all functions are real-valued).
The vector field $b$ is defined by \eqref{b_def}, \eqref{nu_star}.

\begin{theorem}[A priori Hardy inequality]
\label{thm1}
Let $1<\alpha < 2$. 
There exists constant $c=c(d,\alpha)$ such that, for and every $u \in C^\infty=C^\infty(\Pi^d)$,
$$
\big|\bigg\langle b \cdot \nabla u, e^u \bigg\rangle\big| \leq \big\langle (-\Delta)^{\frac{\alpha}{2}}u,e^u\big\rangle + c \big\langle e^u\big\rangle
$$
and, furthermore, for all $\varepsilon>0$,
$$
\big|\bigg\langle b_\varepsilon \cdot \nabla u, e^u \bigg\rangle\big| \leq \big\langle (-\Delta)^{\frac{\alpha}{2}}u,e^u\big\rangle + c \big\langle e^u\big\rangle,
$$
where $b_\varepsilon:=e^{-\varepsilon (-\Delta)^{\frac{\alpha_1}{2}}} b$ (De Giorgi mollifier) with $0<\alpha_1 \leq 2$ fixed by $0 \leq \alpha \leq d-\alpha_1$  possibly after increasing $c$ to $c=c(d,\alpha,\alpha_1)$,.
\end{theorem}

\begin{remark*}
The absolute value in the left-hand side is needed because we do not specify how we extend $b$ to the torus outside of a neighbourhood of the singularity, cf.\,\eqref{H2}. Let us also note that in the case $\alpha=2$ (then necessarily $d \geq 3$) the Hardy inequality of Theorem \ref{thm1} can be obtained from the ordinary Hardy inequality on torus upon integration by parts both in the dispersion term and in the drift term. The gain in the constant in the Hardy inequality compared to the classical argument via Stroock-Varopoulos inequality is a purely non-local phenomenon.
\end{remark*}

Replacing $u$ by $-u$ and adding up the resulting inequalities, one arrives at the inequality
$$
\big|\bigg\langle b \cdot \nabla u, \sinh u \bigg\rangle\big| \leq \big\langle (-\Delta)^{\frac{\alpha}{2}}u,\sinh u\big\rangle + c \big\langle \cosh u\big\rangle.
$$
We will be using Theorem \ref{thm1} to analyze elliptic equation
$$
(\lambda+\Lambda_\varepsilon)u=f, \quad f \in C^\infty,
$$
where
 $$\Lambda_\varepsilon:=(-\Delta)^{\frac{\alpha}{2}} + b_\varepsilon \cdot \nabla,$$
 by actually applying it to $\frac{u}{s}$ for constant $s>0$ that will be chosen sufficiently small, i.e.\,by ``blowing up'' solution $u$. Thus, the term $c \big\langle \cosh u\big\rangle$, although it does not vanish when $u=0$, does not pose a problem. (In the proofs of Corollaries \ref{cor1} and \ref{cor2} we will be using the fact that $\frac{u}{s}$ is still a solution of the Kolmogorov equation with right-hand side $\frac{f}{s}$. In this sense, the test function $u \mapsto e^u$ requires the linearity of the equation.)

Theorem \ref{thm1} yields the following a priori contractivity estimate.
Set
$$\omega:=\frac{c}{2}(1+|\Pi^d|), \quad \text{$c$ is the constant from the Hardy inequality of Theorem \ref{thm1}}.$$

\begin{corollary}[Contractivity in the hyperbolic Orlicz space] 
\label{cor1}
We have
$$
\|e^{-t\Lambda_\varepsilon}f\|_{\cosh -1 } \leq e^{\omega t}\|f\|_{\cosh -1 }, \quad t>0,
$$
for every $\varepsilon>0$.
\end{corollary}

Recall that the Orlicz norm $\|\cdot\|_\Phi$ is defined as
$$
\|v\|_\Phi:=\inf \big\{s>0 \mid \langle \Phi(\frac{v}{s})\rangle \leq 1 \big\}.
$$
Selecting $\Phi(t):=t^p$ recovers the Lebesgue space $L^p$. We take $\Phi(t)=\cosh t - 1$.

We can summarize Theorem \ref{thm1} and Corollary \ref{cor1} by saying that the hyperbolic Orlicz space  ($=$ the completion of $C^\infty(\Pi^d)$ with respect to $\|\cdot\|_{\cosh -1 }$) is an intrinsic space for the fractional Kolmogorov operator. The corresponding intrinsic Hardy inequality is the one contained in Theorem \ref{thm1}. (We ignore for a moment the issue with a priori versus a posteriori variants of these results.)

\medskip

\medskip

For the applications of Theorem \ref{thm1} to the non-local diffusion equation $(\lambda+(-\Delta)^{\frac{\alpha}{2}} + b\cdot \nabla)u=f$ with drift $b=\lim_{\varepsilon \downarrow 0}b_\varepsilon$ given by \eqref{b_def}, \eqref{nu_star}, it is unreasonable to expect smoothness of $u$ (e.g.\,even having control over square integrability of $(-\Delta)^{\frac{\alpha}{2}}u$ is problematic). What regularity can be expected, i.e.\,what estimates survive the passage $\varepsilon \downarrow 0$, is seen from the following a priori estimates. Let $u_\varepsilon$ be solution to the elliptic non-local diffusion equation
$$
(\lambda +  \Lambda_\varepsilon)u_\varepsilon =f, \quad f \in C^\infty.
$$

\begin{enumerate}

\item[--]
The first a priori estimate is straightfroward (in particular, from the probabilistic perspective):
\begin{equation}
\label{disp1}
\tag{$A_1$}
\|u_\varepsilon\|_\infty \leq \frac{1}{\lambda}\|f\|_\infty.
\end{equation}

\item[--]
The second a priori estimate is obtained by multiplying the equation by $u_\varepsilon$ and integrating by parts:
$$
\lambda \langle u_\varepsilon^2\rangle+\langle | (-\Delta)^{\frac{\alpha}{4}}u_\varepsilon|^2\rangle \leq \frac{1}{2} \langle {\rm div\,}b_\varepsilon,u_\varepsilon^2\rangle \leq \frac{\|u_\varepsilon\|_\infty^2}{2} \sup_{\varepsilon>0} \langle  {\rm div\,}b_\varepsilon\rangle.
$$
(Since the coefficient of the drift is large, we cannot use the dispersion term to control  the drift term via the usual fractional Hardy inequality.) So,
\begin{equation}
\label{disp2}
\tag{$A_2$}
\lambda\langle u_\varepsilon^2 \rangle + \langle | (-\Delta)^{\frac{\alpha}{4}}u_\varepsilon|^2\rangle \leq \frac{C}{\lambda^2}\|f\|_\infty^2.
\end{equation}
\end{enumerate}

Thus, to study the non-local diffusion equation with singular drift $b$ and $f \in L^\infty$ (we are interested in probabilistic applications, so condition $f \in L^\infty$ is reasonable), we propose the following a posteriori variant of the Hardy inequality. The constant $c$ is from Theorem \ref{thm1}.

\begin{theorem} [A posteriori Hardy inequality]
\label{thm2}
For every $u \in L^\infty$ such that $ (-\Delta)^{\frac{\alpha}{4}}u \in L^2=L^2(\Pi^d)$,
$$
\big|\bigg\langle T\,(\lambda+ (-\Delta)^{\frac{\alpha}{2}})^{\frac{1}{2}}u, (\lambda+ (-\Delta)^{\frac{\alpha}{2}})^{\frac{1}{2}}e^u \bigg\rangle\big| \leq \big\langle  (-\Delta)^{\frac{\alpha}{4}}u, (-\Delta)^{\frac{\alpha}{4}} e^u\big\rangle + c \langle e^u\rangle,
$$
where $T$ is a bounded operator on $L^2$ given on $C^\infty$ by the formula
$$
T=(\lambda+ (-\Delta)^{\frac{\alpha}{2}})^{-\frac{1}{2}}b \cdot \nabla (\lambda+ (-\Delta)^{\frac{\alpha}{2}})^{-\frac{1}{2}},
$$
with fixed $\lambda \geq \lambda_{d,\alpha}$.
\end{theorem}

As a consequence of Theorem \ref{thm2}, we obtain existence and uniqueness of the weak solution to the elliptic non-local diffusion equation \eqref{el_eq}. In detail:

\begin{definition}
Let $f \in L^\infty$. A function $u \in L^\infty$ such that $ (-\Delta)^{\frac{\alpha}{4}}u \in L^2$ is said to be a weak solution to the elliptic non-local diffusion  equation
\begin{equation}
\label{el_eq}
(\lambda+  (-\Delta)^{\frac{\alpha}{2}} +b\cdot \nabla)u=f,
\end{equation}
with $b$ given by \eqref{b_def}, \eqref{nu_star},
if identity
$$
\lambda\langle u,\varphi\rangle + \langle  (-\Delta)^{\frac{\alpha}{4}}u, (-\Delta)^{\frac{\alpha}{4}}\varphi\rangle + \langle T(\lambda+ (-\Delta)^{\frac{\alpha}{2}})^{\frac{1}{2}}u,(\lambda+A)^{\frac{1}{2}}\varphi\rangle=\langle f,\varphi\rangle,
$$
holds
for all test functions $\varphi \in L^2$ such that $ (-\Delta)^{\frac{\alpha}{4}}\varphi \in L^2$. 
\end{definition}

\begin{corollary}
\label{cor2}
Let $\lambda>c \vee \lambda_{d,\alpha}$. For every $f \in L^\infty$, there exists a unique weak solution to equation \eqref{el_eq}. It satisfies
$$
\|u\|_{\cosh - 1} \leq \frac{1}{\lambda-\omega_0}\|f\|_{\cosh - 1}, \quad \lambda>\omega_0,
$$
where $\omega_0=c \vee \lambda_{d,\alpha} \vee \omega$ (from Corollary \ref{cor1}).
\end{corollary}

\subsection{Remarks} 
1.~The estimate in Corollary \ref{cor2} cannot be obtained by interpolating between  $\|u\|_{\infty} \leq \frac{1}{\lambda}\|f\|_{\infty}$ and an $L^p$ contractivity estimate since for the maximal value of the coupling constant $\nu=\nu_\star$ there is no $L^p$ contractivity estimate to interpolate from. This is in contrast to \cite{BJLP} who always have $L^2$ theory -- Schr\"{o}dinger resolvent and semigroup -- for the maximal value of their coupling constant.

2.~Corollary \ref{cor2} shows that there exists a theory of elliptic equation \eqref{el_eq} with critical drift \eqref{b_def}, \eqref{nu_star}. One can now pose further questions such as whether \eqref{el_eq} determines the resolvent of the generator of a $C_0$ semigroup in the hyperbolic Orlicz space and, furthermore, whether this is a semigroup of integral operators and what are its regularity properties. We hope to address these questions  in subsequent papers.

3.~(Semigroup) Let $\alpha=2$, $d \geq 3$. Armed with the straightforward parabolic analogue of the a priori estimate \eqref{disp2}, the test function $v \mapsto  e^v$ (or $v \mapsto \sinh v$) and the idea of a ``blow up'' of $v$ (which is the idea behind the Orlicz norm),
one can construct an operator realization of $\Lambda=-\Delta + b \cdot \nabla$ as the generator of strongly continuous semigroup in the hyperbolic Orlicz space. The proof consists of showing that solutions $\{v_n\}$ to
$$
(\partial_t - \Delta + b_n \cdot \nabla)v_n=0, \quad v_n|_{t=0}=f \in L^\infty \cap L^1, \quad b_n:=b_{\varepsilon_n},\quad \varepsilon_n \downarrow 0,
$$
constitute a Cauchy sequence:
$$
\sup_{t \in [0,1]}\|v_n(t,\cdot)-v_m(t,\cdot)\|_{\cosh - 1} \rightarrow 0 \quad \text{ as } n,m \rightarrow \infty.
$$
The latter requires showing that 
\begin{equation}
\label{conv9}
\int_0^t \langle (b_n-b_m)\cdot \nabla v_{n}, e^{v_n-v_m}\rangle \rightarrow 0,
\end{equation}
which follows easily upon applying Cauchy-Schwarz and using $\|b_n-b_m\|_2 \rightarrow \infty$, $\sup_{n,m}\|v_n-v_m\|_\infty<\infty$. See \cite{Ki_Orlicz} for details.

In the non-local case $1<\alpha<2$ that argument is not applicable, at least not directly: \eqref{disp2}, or rather its parabolic counterpart, no longer provide an estimate on the gradient of solution in \eqref{conv9}. However, in the sub-critical case $\nu<\nu_\star$ one can construct strongly continuous semigroup for $(-\Delta)^{\frac{\alpha}{2}} + \nu\frac{x}{|x|^\alpha} \cdot \nabla$ in $L^p$, for appropriately chosen $p \uparrow \infty$ as $\nu \uparrow \nu_\star$, using a compactness argument and some auxiliary ``desingularizing'' estimates for this operator; in the critical case $\nu=\nu_\star$ these estimates are not available (remark 5 below).
We cautiously suppose that the following approach to constructing the semigroup when $1<\alpha<2$, $\nu=\nu_\star$ could be useful:  consider approximation  $b_n=(1-1/n)b$, so that every $b_n$ is still singular, but now belongs to the sub-critical regime. The advantage of dealing with thus defined $b_n$ is that one has an analogue of \eqref{conv9} involving operators $T(b_n-b_m)=(\frac{1}{n}-\frac{1}{m})T(b)$ whose norms thus tend to zero as $n,m \rightarrow \infty$ (otherwise, for smooth $b_n$, $b_m$ these operators converge only weakly in $L^2$ or, upon reshuffling $\{b_n\}$ using Mazur's Lemma and the diagonal argument, strongly, cf.\,the proof of Corollary \ref{cor2}, but not in the norm). This approach requires obtaining an analogue of Corollary \ref{cor2} for the parabolic Kolmogorov equation, which seems to be possible with some additional efforts.

On $\mathbb R^d$, in the local case $\alpha=2$, one can still construct the semigroup in the hyperbolic Orlicz space, although under some extra conditions on the decay of drift $b(x)|_{x \in B_1(0)}=\frac{\nu_\star x}{|x|^2}$ at infinity. This required a rather careful work with weights \cite{KS_feller}.

4.~Also on $\mathbb R^d$, in the case $1<\alpha<2$, the test function $u \mapsto \sinh u$ can equally be handled using a Stroock-Varopoulos-type inequality \eqref{SV2}, albeit for $\nu \ll \nu_\star$. Namely:

\begin{enumerate}
\item[--]
 One has the following Stroock-Varopoulos-type inequality:
\begin{equation}
\label{SV2}
\langle (-\Delta)^{\frac{\alpha}{2}}u, \sinh u\rangle \geq 8 \left\langle |(-\Delta)^{\frac{\alpha}{4}} \sinh \frac{u}{2}|^2\right\rangle
\end{equation}
($u$ is assumed to be smooth and vanishing at infinity sufficiently rapidly). We comment on the proof below.

\item[--] Next, by the usual fractional Hardy inequality on $\mathbb R^d$,
\begin{align}
\left\langle |(-\Delta)^{\frac{\alpha}{4}} \sinh \frac{u}{2}|^2\right\rangle & \geq \kappa_{\frac{d-\alpha}{2}}\langle |\cdot|^{-\alpha} (\sinh \frac{u}{2})^2 \rangle \notag \\
& = \frac{1}{2}\kappa_{\frac{d-\alpha}{2}} \langle |\cdot|^{-\alpha} (\cosh u - 1)\rangle.
\label{H4}
\end{align}

\item[--] Consider drift $b(x)=\frac{4}{d-\alpha}\kappa_{\frac{d-\alpha}{2}}\frac{x}{|x|^\alpha} \eta(x)$, where $\eta=1$ in $B_1(0)$, $\eta$ has compact support so that we can freely integrate by parts. We evaluate:
\begin{align*}
|\langle b\cdot \nabla u,\sinh u\rangle | & = |\langle b, \nabla \cosh u\rangle| = |\langle {\rm div\,}b, \cosh u\rangle| \\
& \leq |\langle {\rm div\,}b, \cosh u-1\rangle| + |\langle {\rm div\,}b\rangle| \\
& \leq \langle 4\kappa_{\frac{d-\alpha}{2}}|\cdot|^{-\alpha}, \cosh u-1\rangle + c_0 \langle \cosh u - 1 \rangle + |\langle {\rm div\,}b\rangle|
\end{align*}
(now, instead of the volume of the torus the estimates will contain $|\langle {\rm div\,}b\rangle|$).

\item[--] Finally, using $\langle \lambda u, \sinh u\rangle \geq 2 \lambda \langle \cosh u - 1\rangle$, one can see that the previous estimates can be employed to analyze equation
$$
(\lambda  + (-\Delta)^{\frac{\alpha}{2}} + b \cdot \nabla)u=0 \quad \text{ on } \mathbb R^d,
$$
in the same way as we did it in the proof of Corollary \ref{cor1} (we need to replace $b$ by $b_\varepsilon$, but this is easy to do using Lemma \ref{approx_lem}, or rather its counterpart on $\mathbb R^d$).
\end{enumerate}

The strength of attraction $\nu$ that we can handle in this way is $\nu=\frac{4}{d-\alpha}\kappa_{\frac{d-\alpha}{2}}$. It is strictly less than $\nu_\star$, for all $1<\alpha<2$ (for $\alpha=2$ both coupling constants coincide), which defeats the purpose of dealing with the test function $u \mapsto e^u$ or $u \mapsto \sinh u$ -- one is better off using simply the $L^p$ Hardy inequality \eqref{H2} of \cite{BJLP}. Still, it is satisfying to know that the Hardy inequality of Theorem \ref{thm1} can also be proved, albeit with some essential losses in the condition on $\nu$, using classical means. This also suggests that Theorem \ref{thm1} should admit extension to $\mathbb R^d$.

The  inequality \eqref{SV2} is obtained by following closely the proof of a general Stroock-Varopoulos-type result, \cite[Theorem 2.2]{LS}, once one verifies the following three inequalities for all $t,s \geq 0$:
$$
c_\varphi^{-1}(t-s)(\varphi(t)-\varphi(s)) \geq (G_\varphi(t)-G_\varphi(s))^2,
$$
$$
c_\varphi^{-1}(t+s)(\varphi(t)+\varphi(s)) \geq (G_\varphi(t)+G_\varphi(s))^2,
$$
$$
c_\varphi^{-1}s \varphi(s) \geq G_\varphi^2(s).
$$
for
 $\varphi(u)=e^{u}-e^{-u}$, $G_\varphi(u)=e^{\frac{u}{2}}-e^{-\frac{u}{2}}$ and $c_\varphi=2$. 

Finally, let us add that the Stroock-Varopoulos-type inequality \eqref{SV2} can be used to include (here, for simplicity, at the a priori level, with compact support) a large class of drifts $q:\mathbb R^d \rightarrow \mathbb R^d$ with the divergence satisfying the form-boundedness condition
$$
\langle |{\rm div\,}q|,\varphi^2 \rangle \leq \left\langle |(-\Delta)^{\frac{\alpha}{4}} \varphi|^2\right\rangle
$$
which includes e.g.\,the Morrey class condition $$c_d\sup_{x \in \mathbb R^d, r>0} r^\alpha \biggl(\frac{1}{r^d}\langle \mathbf{1}_{B_r(x)}|{\rm div\,}q|^{1+\varepsilon}\rangle \biggr)^{\frac{1}{1+\varepsilon}} \leq 1,$$see Adams \cite{A}. In this way, we can combine Theorem \ref{thm1} and the argument above to extend Corollary \ref{cor1} to the drifts 
$$
(1-\epsilon)b + 4 \epsilon q, \quad \epsilon>0,
$$
i.e.\,perturb the Hardy drift $b(x)=\frac{\nu_\star x}{|x|^\alpha}$ by a general vector field.

5.~The following was proved in \cite{KSS} regarding the sub-critical regime $\nu<\nu_\star$. The operator $\Lambda_\nu=(-\Delta)^{\frac{\alpha}{2}} + \nu \frac{x}{|x|^\alpha} \cdot \nabla$ generates contraction $C_0$ semigroup $e^{-t\Lambda_\nu}$ in $L^p$ for $p<\infty$ defined by $\nu=\frac{p}{d-\alpha}\kappa_{\frac{d-\alpha}{p}}$ (for $\nu$ small we took $p=2$). This is a semigroup of integral operators whose integral kernel ($=:$ heat kernel) is singular and satisfies two-sided bounds
\begin{align*}
c_1 e^{-t(-\Delta)^{\frac{\alpha}{2}}}(x-y)( t^{-\frac{1}{\alpha}}|y| \wedge 1)^{-d+\gamma} & \leq e^{-t\Lambda_\nu}(x,y) \\
& \leq c_2e^{-t(-\Delta)^{\frac{\alpha}{2}}}(x-y)( t^{-\frac{1}{\alpha}}|y| \wedge 1)^{-d+\gamma}
\end{align*}
for all $t>0, x,y \in \mathbb R^d, y \neq 0$.
Exponent $\gamma \in ]\alpha,d]$ is defined as the unique solution to the equation
\begin{equation}
\label{beta_eq}
\frac{2^\alpha}{\gamma-\alpha}\frac{\Gamma(\frac{\gamma}{2})\Gamma(\frac{d}{2}-\frac{\gamma-\alpha}{2})}{\Gamma(\frac{d}{2}-\frac{\gamma}{2})\Gamma(\frac{\gamma-\alpha}{2})}=\nu.
\end{equation}
The latter entails that $x \mapsto |x|^{-d+\gamma}$ is a Lyapunov function for the formal adjoint operator $\Lambda_\nu^*=(-\Delta)^{\frac{\alpha}{2}}-\nabla \cdot \frac{\nu x}{|x|^\alpha}$, i.e.\,$$\Lambda_\nu^*|\cdot|^{-d+\gamma}=0.$$ 
The proof of the two-sided bounds is based on the $L^1(\mathbb R^d,\varphi_s(y) dy) \rightarrow L^\infty(\mathbb R^d,dx)$ ultracontractivity of $e^{t\Lambda_\nu}$, where $\varphi_s(y):=(s^{-\frac{1}{\alpha}}|y| \wedge 1)^{-d+\gamma}$. In turn, the proof of the weighted ultracontractivity requires the following ``desingularizing $L^1$ bound'': $$\|\varphi_s e^{-t\Lambda_\nu}\varphi_s^{-1}f\|_{1} \leq e^{c_0\frac{t}{s}}\|f\|_{1}, \quad 0<t \leq s, \quad c_0>0$$ 
(or rather its a priori variant for a regularization of the drift). One substantial difficulty here is the non-locality of the operator that makes algebraic manipulations with regularized $\Lambda_\nu$ and regularized $\varphi_s$ rather non-trivial; this was addressed by working with a regularized-drift dependent regularization of $\varphi_s$.
The desingularizing $L^1$ bound in its a priori form was also used in \cite{KSS} to construct the semigroup $e^{-t\Lambda_\nu}$. Let us also emphasize that a natural approach to the proof of heat kernel bounds when $\alpha=2$ via Dirichlet forms is not available for $\Lambda$ in the non-local case $1<\alpha<2$.

The constants $c_0$, $c_1$, $c_2$ depend on the Sobolev embedding carried out by $\Lambda$ and thus degenerate $c_1 \downarrow 0$, $c_0,c_2 \uparrow \infty$ as $\nu \uparrow \nu_\star$.

We note that equation \eqref{beta_eq} also provides an informal evidence that $\nu_\star$ is the maximal coupling constant for Kolmogorov operator \eqref{kolm}, \eqref{b_def}. Indeed, as $\nu \uparrow \nu_\star$, we have $\gamma=\gamma(\nu) \uparrow \alpha$, see Figure \ref{fig3}. For $\gamma>\alpha$ (i.e.\,intuitively, $\nu>\nu_\star$), equation \eqref{beta_eq} and thus the Lyapunov function $|\cdot|^{-d+\gamma}$ cease to be well defined. 

(In the context of this remark, let us also mention that another approach to proving heat kernel bounds based on the spherical harmonics-expansion is developed in recent paper \cite{BM}.)

\begin{figure}[ht]
\includegraphics[scale=0.5]{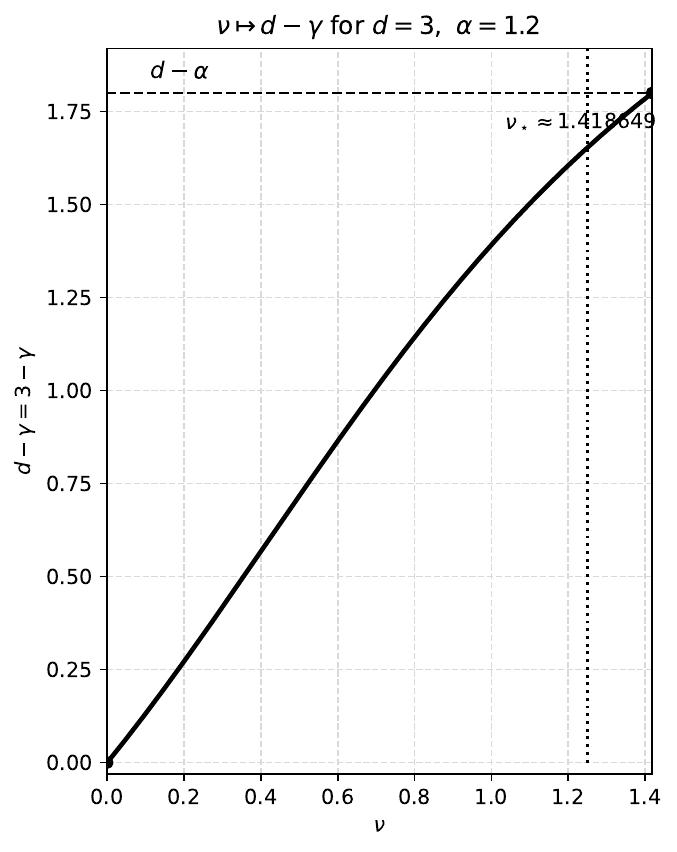}
\label{fig3}
\caption{{\small The dependence of the exponent in the Lyapunov function $|\cdot|^{-d+\gamma}$ on the coupling constant (strength of attraction by the drift) $\nu \in [0,\nu_\star[$. The vertical line represents the largest value of the coupling constant $$\nu_{{\rm SV}}=\lim_{p \uparrow \infty} \frac{p}{d-\alpha} \frac{4(p-1)}{p^2}\kappa_{\frac{d-\alpha}{2}}=\frac{2^{\alpha+2}}{d-\alpha}\left(\frac{\Gamma(\frac{d+\alpha}{4})}{\Gamma(\frac{d-\alpha}{4})}\right)^2$$ that can be attained using the Stroock-Varopoulos inequality \eqref{SV}. It is seen that nothing dramatic happens to the Lyapunov function as $\nu$ reaches and surpasses $\nu_{{\rm SV}}$.}}
\end{figure}

6.~The observation that working in $L^p$ with $p$ chosen sufficiently large allows to relax the condition on the coupling constant of the drift goes back to \cite{KS}. It was further exploited quite substantially in the context of Nash's method for $-\nabla \cdot a \cdot \nabla + q \cdot \nabla$ in \cite{S}.

\subsection{Further discussion}

(\textit{i}) The hyperbolic Orlicz space allows to handle functions with logarithmic singularities. This reminds somewhat the work that been done on local and non-local PDEs with divergence-free singular drifts arising in the study of Navier-Stokes equations, connected to the  ${\rm BMO}$ space, see \cite{CV, CW, KT}. The hyperbolic Orlicz space, unlike ${\rm BMO}$, does not take into account possible cancellations, and the blow-up effects that we encounter in connection with the attracting singularities of the drift have a different, more straightforward nature compared to the blow-ups arising in the study of solutions of the Navier-Stokes equations.

(\textit{ii}) There is a weak solution theory (weak in the probabilistic sense) of the stochastic differential equation (SDE) 
\begin{equation}
\label{sde}
X_t=x-\int_0^t b(X_s)ds + Z_t, \quad x \in \mathbb R^d,
\end{equation}
where $Z_t$ is the isotropic $\alpha$-stable process in $\mathbb R^d$ generated by $(-\Delta)^{\frac{\alpha}{2}}$,
with singular drift $b:\mathbb R^d \rightarrow \mathbb R^d$ in a large class that contains \eqref{kolm}. Namely, $|b| \in L^{1}_{\loc}$ with sufficiently small weak form-bound $\delta$,
\begin{equation}
\label{fbd}
\||b|^{\frac{1}{2}}(\lambda+(-\Delta)^{\frac{\alpha}{2}})^{-\frac{\alpha-1}{2\alpha}}\|_{2 \rightarrow 2} \leq \sqrt{\delta} \quad \text{ for some $\lambda>0$},
\end{equation}
which includes e.g. the case of sufficiently small Morrey norm
$\sup_{x \in \mathbb R^d, r>0} r\bigl(\frac{1}{r^d}\langle \mathbf{1}_{B_r(x)}|b|^{\frac{1+\varepsilon}{\alpha-1}}\rangle \bigr)^{\frac{1}{1+\varepsilon}},$
see \cite{KM, KY}.
Applied to $b(x)=\frac{\nu x}{|x|^\alpha}$, these results require smallness of $\nu$, so
``$\nu<\nu_\star$'' has not been reached yet, except in the case $\alpha=2$ where one can run De Giorgi's method in $L^p$ for $p$ large to prove weak existence for SDE \eqref{sde} and construct the corresponding Feller semigroup in $C_\infty$ for all $\nu<\nu_\star$ (in fact, more generally, for all $b$ satisfying $\|b(\lambda-\Delta)^{-\frac{1}{2}}\|_{2 \rightarrow 2} \leq \delta$ with $\delta$ going up to the critical value) \cite{KS_feller}.

(\textit{iii}) Although the theory of the Kolmogorov operator $\Lambda$ ceases to exist if the coupling constant $\nu$ surpasses $\nu_\star$, this does not mean that there is no theory of isotropic $\alpha$-stable process with drift $\frac{\nu x}{|x|^\alpha}$ . For instance, in the case $\alpha=2$, the authors of \cite{FJ} replace SDE \eqref{sde} with polar drift $b(x)=\nu x/|x|^2$ (the special structure of the drift is important) by the SDE for $|X_t|^2X_t$, which allows them to construct the ``two-particle'' process in the supercritical regime $\nu_\star<\nu~(<\nu_{\star\star})$ that still has physical meaning. In this regard, see also \cite{ORT}.

(\textit{iv}) Regarding the theory of operator $(-\Delta)^{\frac{\alpha}{2}} + \frac{\nu x}{|x|^\alpha} \cdot \nabla$ in regime $\nu<\nu_\star$, we also refer to recent results in \cite{BDM} on the equivalence of Sobolev norms determined by this operator and its fractional powers.

\bigskip

\section{Notations and auxiliary results used in the proofs}
\label{prelim_sect}

1.~In what follows,
\begin{equation}
\label{beta_def}
0 \leq \beta < d-\alpha.
\end{equation}

\medskip

2.~We identify functions on the torus $\Pi^d$ and the $4$-periodic functions on $\mathbb R^d$ -- the Euclidean space is tiled by the translations of the cube $Q=(-2,2]^d$. Having this identification in mind, we define on $\Pi^d$ the following positive function with singularity at the origin (and smooth everywhere else on the torus); it should be viewed as an analogue of function $x \mapsto |x|^{-\beta}$ on $\mathbb R^d$ modulo addition of a bounded function:
$$
\varphi_\beta(x)=\left\{
\begin{array}{ll}
\eta_\beta(t), & x \in B_\frac{3}{2}(0), \\
1, & x \in Q - B_{\frac{3}{2}}(0),
\end{array}
\right. \qquad t=|x|,
$$
where $$\eta_\beta(t):=t^{-\beta a(t)},$$
with a fixed function $a \in C^\infty(]0,2[)$ satisfying
$$
a(t)=\left\{
\begin{array}{ll}
1,& 0<t \leq 1, \\
0, & \frac{3}{2} \leq t <2.
\end{array}
\right.
$$
We obtain right away

\begin{lemma} The following scaling property is valid:
\begin{equation}
\label{scaling}
\varphi_\beta^{p}=\varphi_{p\beta}.
\end{equation}
\end{lemma}

Also, it is clear from the definition that there exists constant $c_0>0$, independent of $0 \leq \beta < d-\alpha$, such that
$$
\varphi_\beta \geq c_0 \quad \text{ on } \Pi^d.
$$

\medskip

3.~It will still be convenient to distinguish between $\varphi_\beta$ and its extension to $\mathbb R^d$ by periodicity, which we denote by $\tilde{\varphi}_\beta$.

\medskip

4.~Set $$
g_\beta:=\tilde{\varphi}_\beta-|\cdot|^{-\beta}.
$$ 
Let us prove some estimates involving $g_\beta$ that will be needed later. 

-- We have $g_\beta=0$ in $B_1(0)$ and, furthermore,
\begin{align}
\|g_\beta\|_{L^\infty(\frac{8}{7}Q)} & =|\tilde{\varphi}_\beta(y)-|y|^{-\beta}|\upharpoonright \text{$y$ is in the corners of $\frac{8}{7} Q$} \notag \\
& = 1-(\frac{16}{7}\sqrt{d})^{-\beta} = 1-e^{-\beta \log (\frac{16}{7}\sqrt{d})} \leq C \beta, \quad 0 \leq \beta < d-\alpha. \label{g_est1}
\end{align}

-- We have $\mathbb R^d=\cup_{m \in 4\mathbb Z^d}(Q+m)$. On $Q+m$ when $m \neq 0$,  $g_\beta$ can be singular, but we can still control its $L^1$ norm, i.e.\,for every $m \in 4\mathbb Z^d$, $m \neq 0$,
\begin{align*}
\|g_\beta\|_{L^1(Q+m)} \equiv \|\tilde{\varphi}_\beta-|\cdot|^{-\beta}\|_{L^1(Q+m)} & = \|\varphi_\beta-|\cdot + m|^{-\beta}\|_{L^1(Q)} \\
& \leq \|\varphi_\beta-1\|_{L^1(Q)} + \|1-|\cdot + m|^{-\beta}\|_{L^1(Q)},
\end{align*}
where, using the spherical coordinates, 
\begin{align*}
\|\varphi_\beta-1\|_{L^1(Q)} & =\|\varphi_\beta-1\|_{L^1(B_1(0))} + \|\varphi_\beta-1\|_{L^1(B_\frac{3}{2}(0)-B_1(0))}\\
& = \omega_d \frac{\beta}{d(d-\beta)} + \omega_d \int_1^\frac{3}{2} r^{d-1}(1-r^{-\beta a(r)})dr.
\end{align*}
In the second integral variable $r$ is bounded away from zero, so representing $r^{-\beta a(r)}=e^{-\beta a(r)\log r }$ and using the Taylor series expansion, we obtain that the second integral is bounded from above by $C\beta$, $0 \leq \beta < d-\alpha$. Thus,
$$
\|\varphi_\beta-1\|_{L^1(Q)} \leq 2C \beta.
$$
Similar analysis yields $\|1-|\cdot + m|^{-\beta}\|_{L^1(Q)} \leq C\beta$ for $C$ independent of $m \neq 0$. Therefore (re-denoting $C$),
\begin{equation}
\label{g_est2}
\|g_\beta\|_{L^1(Q+m)} \leq C \beta, \quad 0 \leq \beta < d-\alpha,
\end{equation}
for $C \neq C(m)$, for all $m \in 4\mathbb Z^d$, $m \neq 0$.

\medskip

5. Put $$A:=(-\Delta)_{\Pi^d}^{\frac{\alpha}{2}}.$$ Then, by definition,
$$
e^{-tA}\varphi_\beta=e^{-t(-\Delta)^{\frac{\alpha}{2}}_{\mathbb R^d}}\tilde{\varphi}_\beta.
$$

\medskip

6. We have, for all $x,y \in \Pi^d$,
$$
e^{-tA}(x-y)=e^{-t(-\Delta)^{\frac{\alpha}{2}}_{\mathbb R^d}}(x-y) + \sum_{m \in 4\mathbb Z^d, m \neq 0} e^{-t(-\Delta)^{\frac{\alpha}{2}}_{\mathbb R^d}}(x-y + m).
$$
So, the action of $e^{-tA}$ on a function on the torus reduces to integrating $f(\cdot)$ against $e^{-t(-\Delta)^{\frac{\alpha}{2}}_{\mathbb R^d}}(x-\cdot)$ plus the integrals of $f(\cdot)$ against the correction terms $\sum_{m \in 4\mathbb Z^d, m \neq 0} e^{-t(-\Delta)^{\frac{\alpha}{2}}_{\mathbb R^d}}(x- \cdot + m)$.

\medskip

7. In the same way, for all $x,y \in \Pi^d$, for $f$ sufficiently smooth,
$$
Af(x)=c_{d,\alpha}\,{\rm p.v.\,}\int_{\Pi^d}(f(x)-f(y))K(x-y) dy,
$$
where
$$
K_\alpha(x-y):=\frac{1}{|x-y|^{d+\alpha}} + \sum_{m \in 4\mathbb Z^d, m \neq 0}  \frac{1}{|x-y+m|^{d+\alpha}}.
$$
We obtain in the usual way the identity
\begin{equation}
\label{torus_frac_form}
\langle A^{\frac{1}{2}}f,A^{\frac{1}{2}}f\rangle = c_{d,\alpha}\int_{\Pi^d} \int_{\Pi^d} |f(x)-f(y)|^2 K_\alpha(x-y) dxdy.
\end{equation}

8.\,We write $L^p=L^p(\Pi^d)$. Let $\|\cdot\|_{p \rightarrow p}$ denote $L^p \rightarrow L^p$ operator norm.

\bigskip

\section{Proof of Theorem \ref{thm1}}

We start with proving a torus analogue of the Hardy inequality of \cite{BJLP}.

\begin{proposition} 
\label{hardy_schr}
There exists constant $c=c(d,\alpha)$ independent of $p$ such that, for every $1<p<\infty$ and all $0 \leq u \in C^\infty(\Pi^d)$,
$$
\kappa_{\frac{d-\alpha}{p}}\big\langle \varphi_{\alpha}, u^p \big\rangle \leq \big\langle Au,u^{p-1}\big\rangle + \frac{c}{p} \langle u^p\rangle.
$$
\end{proposition}

\begin{remarks}
1. It is easy to see that we need to divide $c$ by $p$, e.g.\,take $\alpha=2$ and integrate by parts.

2. The condition $u \geq 0$ is only for convenience: we will be applying the previous inequality to $u=1+\frac{v}{p}$ for $p \rightarrow \infty$. Here $v$ is in general sign-changing, but smooth, so $u$ is positive starting with some $p$.

3. In \cite{BJLP} the authors obtain the corresponding Hardy identity, i.e.\,there is an extra term that admits an explicit description. We follow their argument, so we can also obtain this extra term (from $J_t$, see below).
\end{remarks}

\begin{proof}
Since $u$ is bounded and the semigroup $e^{-tA}$ is strongly continuous on $L^p(\Pi^d)$, we have
\begin{equation}
\label{A_lim}
\big\langle Au,u^{p-1}\big\rangle=\lim_{t \downarrow 0} \big\langle \frac{u-e^{-tA}u}{t},u^{p-1}\big\rangle.
\end{equation}

\begin{remark}
Actually, in the end we will need to consider only two values of $\beta$, i.e. $\frac{d-\alpha}{p}$ and $(p-1)\frac{d-\alpha}{p}$.
\end{remark}

Write $u=v\varphi_\beta$ (then, clearly, $v \geq 0$ vanishes at the origin to order $\beta$) and
$$
\big\langle \frac{u-e^{-tA}u}{t},u^{p-1}\big\rangle=I_t+I_t^\ast + J_t,
$$
where
$$
I_t:=\frac{p-1}{p}\langle v \frac{\varphi_\beta-e^{-tA}\varphi_\beta}{t},(v\varphi_\beta)^{p-1}\rangle,
$$
$$
I_t^\ast:=\frac{1}{p}\langle v^{p-1} \frac{\varphi^{p-1}_\beta-e^{-tA}\varphi^{p-1}_\beta}{t},(v\varphi_\beta)\rangle,
$$
$$
J_t=\frac{1}{p}\langle v^{p-1}\frac{e^{-tA}\varphi_\beta^{p-1}}{t},v\varphi_\beta\rangle + \frac{p-1}{p}\langle v\frac{e^{-tA}\varphi_\beta}{t},(v\varphi_\beta)^{p-1}\rangle - \langle \frac{e^{-tA}(v\varphi_\beta)}{t},(v\varphi_\beta)^{p-1}\rangle.
$$
Then $J_t \geq 0$. In fact, 
$$
J_t=\frac{1}{p}\bigg\langle \bigg\langle F_p(v(x),v(y))\varphi_\beta^{p-1}(x)\varphi_\beta(y)\frac{e^{-tA}(x-y)}{t}\bigg\rangle_y \bigg\rangle_x,
$$
where $$F_p(a,b):=|b|^p-|a|^p-pa^{p-1}(b-a), \quad a,b \in \mathbb R, \quad a^{p-1}=a |a|^{p-2}$$
is the Bregman divergence. It is always non-negative. So, by \eqref{A_lim}, to prove the proposition we need to show that
\begin{equation}
\label{sought_est}
\limsup_{t \downarrow 0}(I_t + I_t^\ast) \geq \kappa_{\frac{d-\alpha}{p}}\big\langle \varphi_{\alpha}, (v\varphi_\beta)^p \big\rangle - \frac{c}{p} \langle (v\varphi_\beta)^p\rangle
\end{equation}
for some $c$ independent of $p$.

\textit{Step 1.~}Let us deal with $I_t$. We have, for every $x \in Q$,
\begin{align*}
\frac{\varphi_\beta(x)-(e^{-tA}\varphi_\beta)(x)}{t} &=\frac{\tilde{\varphi}_\beta(x)-(e^{-t(-\Delta)^{\frac{\alpha}{2}}_{\mathbb R^d}}\tilde{\varphi}_\beta)(x)}{t} \\
& (\text{write } e^{-t(-\Delta)^{\frac{\alpha}{2}}_{\mathbb R^d}}\tilde{\varphi}_\beta=e^{-t(-\Delta)^{\frac{\alpha}{2}}_{\mathbb R^d}}|\cdot|^{-\beta} + e^{-t(-\Delta)^{\frac{\alpha}{2}}_{\mathbb R^d}}(\tilde{\varphi}_\beta-|\cdot|^{-\beta})) \\
&=\frac{|x|^{-\beta}-(e^{-t(-\Delta)^{\frac{\alpha}{2}}_{\mathbb R^d}}|\cdot|^{-\beta})(x)}{t} + \frac{g_\beta(x) - (e^{-t(-\Delta)^{\frac{\alpha}{2}}}g_\beta)(x)}{t},
\end{align*}
where, recall, 
$$
g_\beta=\tilde{\varphi}_\beta-|\cdot|^{-\beta}.
$$ 
So, we can write
\begin{align*}
\frac{p}{p-1}I_t & = \left\langle \frac{u^p(x)}{\varphi_\beta(x)},\frac{|x|^{-\beta}-(e^{-t(-\Delta)^{\frac{\alpha}{2}}_{\mathbb R^d}}|\cdot|^{-\beta})(x)}{t}\right\rangle_{x \in Q}  + \left\langle \frac{u^p(x)}{\varphi_\beta(x)}, \frac{g_\beta(x) - (e^{-t(-\Delta)^{\frac{\alpha}{2}}}g_\beta)(x)}{t}\right\rangle_{x \in Q} \\
&=:S_t^{1} + S_t^2.
\end{align*}
We now proceed to estimating $S_t^1$ and $S_t^2$.

\medskip

(a) Regarding $S_t^1$, we have
\begin{align}
\lim_{t \downarrow 0}S_t^1 & =\kappa_\beta  \langle u^p(x) \frac{|x|^{-\beta-\alpha}}{\varphi_\beta(x)} \rangle_{x \in Q} \notag \\ 
& \geq \kappa_\beta \langle u^p\varphi_\alpha \rangle_{\Pi_d} - \kappa_\beta  c_1 \langle u^p \rangle_{\Pi^d}, \label{S1}
\end{align}
where constant $c_1=\max_{x \in Q}\bigl(\frac{|x|^{-\alpha-\beta}}{\varphi_\beta(x)}-\varphi_{\alpha}(x)\bigr)<\infty$.
The inequality is straightforward.
To show the claimed convergence, we note that the first factor in $S_t^1$ is bounded and vanishes to order $\beta$ at the origin, while the second factor contains integration over $\mathbb R^d$, so we can treat it using an argument from \cite{BJLP, BDK}. Namely, one has
\begin{equation}
\label{lyapunov_rep}
|x|^{-\beta}=\int_0^\infty f_\beta(s)e^{-s(-\Delta)_{\mathbb R^d}^{\frac{\alpha}{2}}}(x)ds, 
\end{equation}
$$
\kappa_\beta |x|^{-\beta-\alpha}=\int_0^\infty f_\beta'(s)e^{-s(-\Delta)_{\mathbb R^d}^{\frac{\alpha}{2}}}(x)ds,
$$
where 
$$
f_\beta(s)=c_0\left\{
\begin{array}{l}
s^{\frac{d-\alpha-\beta}{\alpha}}, \quad s>0, \\
0, \quad s \leq 0
\end{array}
\right.
$$
for appropriately chosen constant $c_0$. Then
$$
\frac{|x|^{-\beta}-(e^{-t(-\Delta)^{\frac{\alpha}{2}}_{\mathbb R^d}}|\cdot|^{-\beta})(x)}{t}=\int_0^\infty \frac{f(s)-f(s-t)}{t} e^{-s(-\Delta)_{\mathbb R^d}^{\frac{\alpha}{2}}}(x)ds.
$$
It remains to note that $0 \leq f(s)-f(s-t) \leq Ctf'(t)$, for $s,t>0$, so the convergence follows upon applying the Dominated Convergence Theorem.

\begin{remark}
If we were to apply the previous construction directly on the torus, we would need to introduce a decreasing exponential factor in \eqref{lyapunov_rep} to make the integral converge. This makes the argument considerably lengthier compared to $\mathbb R^d$, so we opt for the present proof.
\end{remark}

(b) In $S_t^2$, the function $g_\beta$, defined on $\mathbb R^d$, is smooth and vanishes in $B_1(0)$. Fix a bump function $0 \leq \zeta \leq 1$ such that $\sprt \zeta = 1$ on $\frac{15}{14}Q$, $\sprt \zeta \subset \frac{8}{7}Q$. We represent 
$$
g_\beta=\zeta g_{\beta}+(1-\zeta)g_\beta,
$$
where $\zeta g_{\beta}$ thus has compact support in $\frac{8}{7}Q - B_1(0)$ and $$(1-\zeta)g_\beta=0 \quad \text{ on $\frac{15}{14}Q$.}
$$ Then
$$
S_t^2=\left\langle \frac{u^p(x)}{\varphi_\beta(x)},\frac{(\zeta g_{\beta})(x)-(e^{-t(-\Delta)^{\frac{\alpha}{2}}_{\mathbb R^d}}\zeta g_{\beta})(x)}{t}\right\rangle_{x \in Q}  + \left\langle \frac{u^p(x)}{\varphi_\beta(x)},-\frac{(e^{-t(-\Delta)^{\frac{\alpha}{2}}_{\mathbb R^d}}(1-\zeta) g_{\beta})(x)}{t}\right\rangle_{x \in Q}
$$
In the first term, the first factor is a continuous function while $\zeta g_\beta$ is a smooth function having compact support, so we have
\begin{align*}
\left\langle \frac{u^p(x)}{\varphi_\beta(x)},\frac{(\zeta g_{\beta})(x)-(e^{-t(-\Delta)^{\frac{\alpha}{2}}_{\mathbb R^d}}\zeta g_{\beta})(x)}{t}\right\rangle_{x \in Q} \underset{t \downarrow 0}{\rightarrow } & \left\langle \frac{u^p(x)}{\varphi_\beta(x)} (-\Delta)^{\frac{\alpha}{2}}_{\mathbb R^d} (\zeta g_\beta)(x)\right\rangle_{x \in Q} \\
& \geq - c_0^{-1}c_2 \beta \langle u^p\rangle_{\Pi^d}
\end{align*}
using  $\varphi_\beta \geq c_0$ and
$$
\|(-\Delta)^{\frac{\alpha}{2}}_{\mathbb R^d} (\zeta g_\beta)\|_{L^\infty(\mathbb R^d)} \leq c_2\beta.
$$
To see the latter it suffice to show that $\|\nabla_i \nabla_j\zeta g_\beta\|_{L^\infty(\mathbb R^d)} \leq c_2'\beta$
for some $c_2'$ independent of $\beta$, and then use the principal value representation for the fractional Laplacian and apply the Mean-value Theorem. The estimate on the second derivatives of $\zeta g_\beta$ follows from $\sprt \zeta g_\beta \subset \frac{8}{7}Q - B_1(0)$, i.e.\,are away from the singularities of $g_\beta$, using \eqref{g_est1}.

The second term is the above representation for $S_t^2$ is handled as follows. By $\frac{e^{-t(-\Delta)^{\frac{\alpha}{2}}_{\mathbb R^d}}(x-y)}{t} \leq C |x-y|^{-d-\alpha}$ (from the standard upper bound of the heat kernel of the fractional Laplacian),
\begin{align*}
\left\langle \frac{u^p(x)}{\varphi_\beta(x)},-\frac{(e^{-t(-\Delta)^{\frac{\alpha}{2}}_{\mathbb R^d}}(1-\zeta) g_{\beta})(x)}{t}\right\rangle_{x \in Q} & \geq -C \left\langle \frac{u^p(x)}{\varphi_\beta(x)}, \big\langle |x-y|^{-d-\alpha}(1-\zeta(y)) |g_{\beta}(y)|\big\rangle_{y \in \mathbb R^d}\right\rangle_{x \in Q} \\
& (\text{use } \varphi_\beta \geq c_0 \text{ and } (1-\zeta)=0 \text{ on } \small{\frac{15}{14}}Q) \\
& \geq -C c_0^{-1}\left\langle u^p(x), \big\langle |x-y|^{-d-\alpha}|g_{\beta}(y)|\big\rangle_{y \in \mathbb R^d  \setminus \frac{15}{14}Q}\right\rangle_{x \in Q} \\
& \text{(note that $x$ and $y$ are uniformly bounded away:} \\
& \text{this is the reason we remove $\frac{15}{14}Q$, not just $Q$)} \\
& \geq -C_2 \sum_{m \in 4\mathbb Z^d, m \neq 0}\big\langle |y|^{-d-\alpha}|g_{\beta}(y)|\big\rangle_{y \in Q+m}\left\langle u^p(x) \right\rangle_{x \in Q} \\
& (\text{apply \eqref{g_est2}}) \\
& \geq -  \sum_{m \in 4\mathbb Z^d, m \neq 0} |m|^{-d-\alpha} C\beta \left\langle u^p(x) \right\rangle_{x \in Q} = - c_3 \beta \langle u^p \rangle_{\Pi^d}
\end{align*}
where $c_3<\infty$ is independent of $\beta$.

Thus,
\begin{equation}
\label{S2}
\lim_{t \downarrow 0}S_t^2 \geq -(c_2+c_3)\beta \langle u^p \rangle.
\end{equation}

Combining \eqref{S1} and \eqref{S2}, we arrive at
\begin{equation}
\label{step1}
\lim_{t \downarrow 0}I_t \geq \frac{p-1}{p}\kappa_{\beta}  \langle  \varphi_\alpha, u^p \rangle_{\Pi_d} - \frac{p-1}{p}(\kappa_\beta c_1 + \beta c_2 + \beta c_3) \langle u^p \rangle_{\Pi^d}.
\end{equation}

\medskip

\textit{Step 2.~}Denote $I_t=I_t(v,\varphi_\beta,p)$. Then we can represent, using, crucially, the scaling property \eqref{scaling}, 
$$I_t^\ast=I_t(v^{\frac{p}{p'}},\varphi_{(p-1)\beta},p').$$
Therefore, applying \eqref{step1} to $I_t(v^{\frac{p}{p'}},\varphi_{(p-1)\beta},p')$ and noting that we have $(v^{\frac{p}{p'}}\varphi_{(p-1)\beta})^{p'}=(v\varphi_\beta)^p=u^p$ and $\frac{p'-1}{p'}=\frac{1}{p}$, we obtain
$$
\lim_{t \downarrow 0}I^\ast_t \geq \frac{1}{p}\kappa_{(p-1)\beta}  \langle  \varphi_\alpha, u^p \rangle_{\Pi_d} - \frac{1}{p}c \langle u^p \rangle_{\Pi^d}.
$$

\textit{Step 3}.~Take $\beta=\frac{d-\alpha}{p}$. Then $\kappa_\beta=\kappa_{(p-1)\beta}$. Then, adding up the last inequality with \eqref{step1}, we arrive at the sought estimate \eqref{sought_est} (modulo re-denoting $2c$ by $c$).
\end{proof}

It will be convenient to introduce coefficient-less drift
\begin{equation}
\label{drift2}
q(x)=\frac{x}{|x|^{\alpha}} \quad \text{ if } |x|\leq 1 \quad \text{extended to the torus as \eqref{b_def}},
\end{equation}
in which case the drift that interests us, i.e.\,$b$ in \eqref{b_def}, \eqref{nu_star}, is given by $b=\nu_\star q$.

Fix some $1<p<\infty$. For the Kolmogorov operator 
\begin{equation*}
(-\Delta)^{\frac{\alpha}{2}} + \frac{p}{d-\alpha}\kappa_{\frac{d-\alpha}{p}}\, q \cdot \nabla,
\end{equation*}
(in view of \eqref{max_conv}, $\frac{p}{d-\alpha}\kappa_{\frac{d-\alpha}{p}} q \rightarrow b$ as $p \uparrow \infty$),
we have the following intrinsic $L^p$ Hardy inequality:

\begin{proposition} 
\label{hardy_kolm}
There exists constant $c=c(d,\alpha)$ independent of $p$ such that, for every $1<p<\infty$ and all $0 \leq u \in C^\infty(\Pi^d)$,
$$
\frac{p}{d-\alpha}\kappa_{\frac{d-\alpha}{p}}|\big\langle q \cdot \nabla u, u^{p-1} \big\rangle| \leq \big\langle Au,u^{p-1}\big\rangle + \frac{c}{p} \langle u^p\rangle. 
$$
\end{proposition}
\begin{proof}
We integrate by parts in the drift term:
\begin{align*}
\big\langle q \cdot \nabla u, u^{p-1} \big\rangle = -\frac{1}{p}\big\langle {\rm div\,}q,u^p\big\rangle.
\end{align*}
In turn, by \eqref{drift2},
$$
{\rm div\,}q=(d-\alpha)|x|^{-\alpha} \quad \text{ if } x \in B_1(0), 
$$
and ${\rm div\,}q$ is smooth on $\Pi^d - B_1(0)$. Hence $${\rm div\,}q \leq (d-\alpha)\varphi_\alpha + c_0$$ for a  constant $c_0=c_0(d,\alpha)$. Therefore
\begin{align*}
|\big\langle q \cdot \nabla u, u^{p-1} \big\rangle| &\leq \frac{d-\alpha}{p} \big\langle \varphi_\alpha, u^p \big\rangle + \frac{d-\alpha}{p}c_0\langle u^p\rangle \\
& \text{(we apply Proposition \ref{hardy_schr})} \\
& \leq \frac{d-\alpha}{p}\kappa_{\frac{d-\alpha}{p}}^{-1}\bigg(\langle Au,u^{p-1}\rangle + \frac{c}{p} \langle u^p\rangle\bigg) + \frac{d-\alpha}{p}c_0 \langle u^p\rangle,
\end{align*}
which yields
$$
\frac{p}{d-\alpha}\kappa_{\frac{d-\alpha}{p}}|\big\langle q \cdot \nabla u, u^{p-1} \big\rangle| \leq \langle Au,u^{p-1}\rangle + \biggl( \frac{c}{p}+ c_0\kappa_{\frac{d-\alpha}{p}} \biggr)\langle u^p\rangle.
$$
It remains to take into account that $\kappa_{\frac{d-\alpha}{p}}=O(1/p)$.
\end{proof}

\begin{corollary}
\label{cor_kolm}
There exists constant $c=c(d,\alpha)$ independent of $p$ such that, for every $1<p<\infty$ and all $0 \leq u \in C^\infty(\Pi^d)$,
$$
\frac{p}{d-\alpha}\kappa_{\frac{d-\alpha}{p}}\big|\bigg\langle q \cdot \nabla u, \bigg(1+\frac{u}{p}\bigg)^{p-1} \bigg\rangle\big| \leq \bigg\langle Au,\bigg(1+\frac{u}{p}\bigg)^{p-1}\bigg\rangle + c \big\langle \bigg(1+\frac{u}{p}\bigg)^p\big\rangle. 
$$
\end{corollary}

\begin{proof}
Apply Proposition \ref{hardy_kolm} to $v:=1+\frac{u}{p}$:
$$
\frac{p}{d-\alpha}\kappa_{\frac{d-\alpha}{p}}|\big\langle q \cdot \nabla v, v^{p-1} \big\rangle| \leq \big\langle Av,v^{p-1}\big\rangle + \frac{c}{p} \langle v^p\rangle. 
$$
We have $\nabla v=\frac{u}{p}$, $Av=\frac{u}{p}$. So, it remains to multiply by $p$.
\end{proof}

Proposition \ref{hardy_kolm} and Corollary \ref{cor_kolm} admit a priori variants, i.e.\,for a smooth a.e.\,approximation $\{q_\varepsilon\}$ of $q$ as long as it satisfies
\begin{equation}
\label{approx_dom}
{\rm div\,}q_\varepsilon \leq {\rm div\,}q + C
\end{equation}
for some constant $\sup_{\varepsilon>0}C<\infty$. Then we can repeat the proof of Proposition \ref{hardy_kolm} using $\kappa_{\frac{d-\alpha}{p}}=O(1/p)$).

\begin{lemma}
\label{approx_lem}
One possible choice of $\{q_\varepsilon\}$ is as follows:
$$
q_\varepsilon:=E_{\varepsilon} q,
$$
where $E_{\varepsilon}=e^{-\varepsilon(-\Delta)_{\Pi^d}^{\frac{\alpha_1}{2}}}$ is the De Giorgi mollifier with $0<\alpha_1 \leq 2$ fixed by $0 \leq \alpha \leq d-\alpha_1$.
\end{lemma}

The proof below can be repeated for the norm of $q_\varepsilon$ itself, to yield, in addition to \eqref{approx_dom},
\begin{equation} 
\label{approx_dom2}
|q_\varepsilon| \leq |q|+C
\end{equation}
(we will use this later).

\begin{proof}
The convergence follows from the properties of mollifiers, while the domination property \eqref{approx_dom} follows, once we fix $\alpha_1$ sufficiently close to $0$, from the following super-median property of the $\alpha$-stable semigroup on $\mathbb R^d$:
\begin{equation}
\label{super_median}
(e^{-(-\Delta)^{\frac{\alpha_1}{2}}_{\mathbb R^d}}|\cdot|^{-\beta}\big)(x) \leq |x|^{-\beta}, \quad x \in \mathbb R^d, \quad 0 \leq \beta < d-\alpha_1,
\end{equation}
see \cite{BDK} for details.
Indeed, on the torus $\Pi^d$ we have
\begin{align*}
{\rm div\,}q_\varepsilon=E_\varepsilon {\rm div\,}q=(d-\alpha) E_\varepsilon \varphi_\alpha + c_0,
\end{align*}
for a universal constant $c_0$,
where, in turn, for every $x \in Q$,
\begin{align}
E_\varepsilon \varphi_\alpha(x) & \equiv e^{-\varepsilon(-\Delta)^{\frac{\alpha_1}{2}}_{\Pi^d}}\varphi_\alpha(x)  =e^{-\varepsilon(-\Delta)^{\frac{\alpha_1}{2}}_{\mathbb R^d}}\tilde{\varphi}_\alpha(x) \notag \\
& = \langle e^{-\varepsilon(-\Delta)^{\frac{\alpha_1}{2}}_{\mathbb R^d}}(x-y)\varphi_\alpha(y)\rangle_{y \in 3Q} + \sum_{\substack{m=(m_i)_{i=1}^d \in 4\mathbb Z^d, \\ \text{at least one $|m_i| \geq 2$}}} \langle e^{-\varepsilon(-\Delta)^{\frac{\alpha_1}{2}}_{\mathbb R^d}}(x-y-m)\varphi_\alpha(y)\rangle_{y \in Q}. \label{moll_est}
\end{align}
The first term in \eqref{moll_est}:
\begin{align*}
\langle e^{-\varepsilon(-\Delta)^{\frac{\alpha_1}{2}}_{\mathbb R^d}}(x-y)\varphi_\alpha(y)\rangle_{y \in 3Q} & \leq \langle e^{-\varepsilon(-\Delta)^{\frac{\alpha_1}{2}}_{\mathbb R^d}}(x-y)|y|^{-\alpha}\rangle_{y \in Q} + c_1 \\
& < \langle e^{-\varepsilon(-\Delta)^{\frac{\alpha_1}{2}}_{\mathbb R^d}}(x-y)|y|^{-\alpha}\rangle_{y \in \mathbb R^d} + c_1 \\
& (\text{we apply \eqref{super_median} with $\beta=\alpha$}) \\
& \leq |x|^{-\alpha} \quad \text{provided that $\alpha_1$ is fixed by $0 \leq \alpha \leq d-\alpha_1$} \\
& \leq \varphi_\alpha(x) + c_2.
\end{align*}
(Note that if $d \geq 4$, then we can take $\alpha_1=\alpha$ or even $\alpha_1=2$.)

The second term in  \eqref{moll_est} contains
\begin{align*}
&\langle e^{-\varepsilon(-\Delta)^{\frac{\alpha_1}{2}}_{\mathbb R^d}}(x-y-m)\varphi_\alpha(y)\rangle_{y \in Q} \\
& (\text{apply the standard upper bound on the heat kernel of the fractional Laplacina}) \\
& \leq \varepsilon \langle |x-y-m|^{-d-\alpha_1}\varphi_\alpha(y)\rangle_{y \in Q}.
\end{align*}
The integral kernel is bounded since $x$ and $y$ are separated.
Hence, uniformly in $x \in Q$,
\begin{align*}
\sum_{\substack{m=(m_i)_{i=1}^d \in 4\mathbb Z^d, \\ \text{at least one $|m_i| \geq 2$}}} \langle e^{-\varepsilon(-\Delta)^{\frac{\alpha_1}{2}}_{\mathbb R^d}}(x-y-m)\varphi_\alpha(y)\rangle_{y \in Q} & \leq C_1 \varepsilon  \sum_{\substack{m=(m_i)_{i=1}^d \in 4\mathbb Z^d, \\ \text{at least one $|m_i| \geq 2$}}} m^{-d-\alpha_1}\\
&=:C_2 \varepsilon, \quad C_2<\infty.
\end{align*}

Therefore, we obtain from \eqref{moll_est}:
$$
E_\varepsilon \varphi_\alpha(x) \leq \varphi_\alpha(x)+C_3, \quad C_3=C_2\varepsilon + c_2.
$$
This yields \eqref{approx_dom} with $C=(d-\alpha)C_3+c_0$.
\end{proof}

To end the proof of Theorem \ref{thm1} we take $p \rightarrow \infty$ in Corollary \ref{cor_kolm} and its a priori variant for $q_\varepsilon$ using that $\frac{p}{d-\alpha}\kappa_{\frac{d-\alpha}{p}} q \rightarrow b$, $\frac{p}{d-\alpha}\kappa_{\frac{d-\alpha}{p}} q_\varepsilon \rightarrow b_\varepsilon$ as $p \uparrow \infty$. Since $u$ is assumed to be smooth, the passage to the limit presents no problem. \hfill \qed

\bigskip

\section{Proof of Theorem \ref{thm2}}

Step 1.~In order to show the boundedness of operator $T$ on $L^2$, we apply Stein's Interpolation Theorem to the following analytic family of operators defined on $C^\infty$, a dense subset of $L^2$:
$$
T_z=(\lambda+A)^{-z}b \cdot \nabla (\lambda+A)^{-1+z}, \quad 0 \leq \Real
\,z \leq 1.$$
We have
$$
\|T_{i\gamma}f\|_2 \leq \|(\lambda+A)^{-i\gamma}\|_{2 \rightarrow 2} \|b \cdot \nabla (\lambda+A)^{-1}\|_{2 \rightarrow 2}\|(\lambda+A)^{i\gamma}f\|_2.
$$
The purely imaginary powers are bounded operators on $L^2$ since their Fourier symbols are in $L^\infty$. The boundedness of the operator in the middle follows e.g.\,upon applying the pointwise estimate on the gradient of the resolvent of the fractional Laplacian and then using the usual fractional Hardy inequality on the torus:
$$
\||b|(\lambda_{d,\alpha}+A)^{-1+\frac{1}{\alpha}}\|_{2 \rightarrow 2} \leq c_{d,\alpha},
$$
so we need to assume that $\lambda \geq \lambda_{d,\alpha}>0$. (We do not need sharp constants here.) Thus, $\|T_{i\gamma}f\|_2 \leq C\|f\|_2$ for all $f \in C^\infty$, $\gamma \in \mathbb R$.

Next, we write
$$
T_{1+i\gamma}f=\nabla (\lambda+A)^{-1-i\gamma}b (\lambda+A)^{i\gamma} - (\lambda+A)^{-1-i\gamma}{\rm div\,}b\, (\lambda+A)^{i\gamma}, \quad f \in C^\infty.
$$
Once again, we put the purely imaginary powers outside and then apply the usual fractional Hardy inequality to obtain
$\|T_{1+i\gamma}f\|_2 \leq C\|f\|_2$ for all $f \in C^\infty$, $\gamma \in \mathbb R$.

Stein's Interpolation Theorem yields ($z=\frac{1}{2}$)
$$
\|(\lambda+A)^{-\frac{1}{2}}b \cdot \nabla (\lambda+A)^{-\frac{1}{2}}f \| \leq C\|f\|_2,
$$
and so operator $T$ is bounded. We denote its extension by continuity to $L^2$ also by $T=T(b)$.
Let us note that the norm of this operator will not enter the sought a posteriori Hardy inequality: we only need its boundedness to make sense of the left-hand side of the inequality.

\begin{remark}
\label{dual_rem}
 The same argument yields that
$$
T^\ast(b) f:=-(\lambda+A)^{-\frac{1}{2}}\nabla \cdot b (\lambda+A)^{-\frac{1}{2}}, \quad f \in C^\infty,
$$
is bounded and admits extension by continuity to $L^2$, which we denote also by $T^\ast(b)$; we will use this in the proof of Corollary \ref{cor2}.
\end{remark}

Step 2.~Let $u$ be as in the statement of the theorem, i.e.\,$u \in L^\infty$, $Au \in L^2$. Integrating by parts, we obtain that $(\lambda+A)^{\frac{1}{2}}u \in L^2$. We can find $u_\varepsilon \in C^\infty$, $\sup_{\varepsilon>0}\|u_\varepsilon\|_\infty<\infty$ such that
$$u_\varepsilon \rightarrow u, \quad A^{\frac{1}{2}}u_\varepsilon \rightarrow A^{\frac{1}{2}}u, \quad (\lambda+A)^{\frac{1}{2}}u_\varepsilon \rightarrow (\lambda+A)^{\frac{1}{2}}u, \quad \text{ in } L^2$$ as $\varepsilon \downarrow 0$, e.g.\,take $u_\varepsilon:=e^{-\varepsilon A}u$.

In the next calculation we use estimate
$$
\langle |A^{\frac{1}{2}}e^{u_\varepsilon}|^2\rangle \leq \|e^{u_\varepsilon}\|^2_\infty \langle |A^{\frac{1}{2}}u_\varepsilon|^2\rangle,
$$
which is an immediate consequence of the quadratic form representation \eqref{torus_frac_form} and of the Mean-value Theorem.
We have
\begin{align*}
\langle (\lambda+A)^{\frac{1}{2}}e^{u_\varepsilon},(\lambda+A)^{\frac{1}{2}}e^{u_\varepsilon}\rangle & = \lambda \langle e^{2u_\varepsilon}\rangle + \langle A^{\frac{1}{2}}e^{u_\varepsilon},A^{\frac{1}{2}}e^{u_\varepsilon}\rangle \\
& \leq \lambda \langle e^{2u_\varepsilon}\rangle + \|e^{u_\varepsilon}\|_\infty^2 \langle A^{\frac{1}{2}}u_\varepsilon,A^{\frac{1}{2}}u_\varepsilon\rangle.
\end{align*}
Hence, since $u \in L^\infty$ and $\sup_{\varepsilon>0}\|u_\varepsilon\|_\infty<\infty$, we have $\sup_{\varepsilon>0}\|(\lambda+A)^{\frac{1}{2}}e^{u_\varepsilon}\|_2<\infty$.
The latter and the fact that $e^{u_\varepsilon} \rightarrow e^u$ in $L^2$ (again, use the boundedness of $u$) yield
$$
(\lambda+A)^{\frac{1}{2}}e^{u_\varepsilon} \rightarrow (\lambda+A)^{\frac{1}{2}}e^{u} \quad \text{weakly in } L^2.
$$
Above one can take $\lambda=0$ (we need strictly positive $\lambda$ only when dealing with operator $T$). 

It follows that
\begin{align*}
\langle A^{\frac{1}{2}}u_\varepsilon,A^{\frac{1}{2}}e^{u_\varepsilon}\rangle & = \langle A^{\frac{1}{2}}(u_\varepsilon-u),A^{\frac{1}{2}}e^{u_\varepsilon}\rangle \\
& + \langle A^{\frac{1}{2}}u,A^{\frac{1}{2}}(e^{u_\varepsilon}-e^u)\rangle + \langle A^{\frac{1}{2}}u,A^{\frac{1}{2}}e^u\rangle \\
& \rightarrow \langle A^{\frac{1}{2}}u,A^{\frac{1}{2}}e^u\rangle
\end{align*}
and
\begin{align*}
\langle T(\lambda+A)^{\frac{1}{2}}u_\varepsilon,(\lambda+A)^{\frac{1}{2}}e^{u_\varepsilon}\rangle & = \langle T(\lambda+A)^{\frac{1}{2}}(u_\varepsilon-u),(\lambda+A)^{\frac{1}{2}}e^{u_\varepsilon}\rangle \\
& + \langle T(\lambda+A)^{\frac{1}{2}}u,(\lambda+A)^{\frac{1}{2}}(e^{u_\varepsilon}-e^u)\rangle + \langle T(\lambda+A)^{\frac{1}{2}}u,(\lambda+A)^{\frac{1}{2}}e^u\rangle \\
& \rightarrow \langle T(\lambda+A)^{\frac{1}{2}}u,(\lambda+A)^{\frac{1}{2}}e^u\rangle.
\end{align*}
The result now follows from the a priori Hardy inequality of Theorem \ref{thm1} upon taking $\varepsilon \downarrow 0$.
\hfill \qed

\bigskip

\section{Proof of Corollary \ref{cor1}}
Let $v=v_\varepsilon$ be the solution to 
$
\big(\lambda+\partial_t+ \Lambda_\varepsilon\big)v=0$, $v|_{t=0}=f,
$
where $\lambda$ is fixed by $2\lambda=c$.
Our aim is to prove bound
\begin{equation}
\label{v_bd}
\|v(t,\cdot)\|_{\cosh-1} \leq e^{\frac{c}{2}|\Pi^d|t}\|f\|_{\cosh -1}, \quad t>0.
\end{equation}

Step 1.~We multiply the parabolic equation by $e^v$ and integrate:
$$
\lambda \langle v,e^v \rangle + \langle \partial_t v,e^v \rangle + \langle Av,e^v\rangle + \langle b_\varepsilon \cdot \nabla v,e^v\rangle=0.
$$
 Theorem \ref{thm1} yields
$$
\lambda \langle v,e^v \rangle + \langle \partial_t v,e^v \rangle \leq c\langle e^v \rangle,
$$
and so
$
\lambda \langle v,e^v \rangle + \langle \partial_t (e^v-1)\rangle \leq c\langle e^v \rangle.
$
Therefore, integrating in time, we arrive at
\begin{equation*}
\lambda \int_0^t \langle v,e^{v}\rangle + \langle e^{v(t)}-1 \rangle \leq \langle e^{f}-1 \rangle + c \int_0^t \langle e^{v}\rangle. 
\end{equation*}
Replacing in the last inequality $v$ by $-v$ and adding up the resulting inequalities, we obtain
\begin{equation*}
\lambda  \int_0^t \langle v\sinh(v) \rangle  + \langle \cosh (v(t))-1\rangle \leq \langle \cosh (f)-1\rangle + c \int_0^t \langle \cosh(v)\rangle.
\end{equation*}
Applying  $v\sinh(v) \geq 2(\cosh(v)-1)$, we arrive at 
\begin{equation}
\label{e3}
(2\lambda-c) \int_0^t \langle \cosh(v)-1 \rangle + \langle \cosh (v(t))-1\rangle \leq \langle \cosh (f)-1\rangle + c |\Pi^d| t.
\end{equation}
The latter is a non-linear inequality in $v$. However, $v$ solves a linear equation, therefore we can ``blow up'' inequality \eqref{e3}, i.e.\,replace solution $v$ by $\frac{v}{s}$, $s>0$, with the goal to later select $s$ sufficiently small. Crucially, introducing $s$ does not affect the last term  $c |\Pi^d| t$ since $\frac{v}{s}$ is just another solution:
\begin{equation*}
(2\lambda-c) \int_0^t \langle \cosh(\frac{v}{s})-1 \rangle  + \langle \cosh (\frac{v(t)}{s})-1\rangle \leq \langle \cosh (\frac{f}{s})-1\rangle + c|\Pi^d| t.
\end{equation*}
Recalling that we have chosen $\lambda=\frac{c}{2}$, we have
\begin{equation*}
\langle \cosh (\frac{v(t)}{s})-1\rangle \leq \langle \cosh (\frac{f}{s})-1\rangle + c|\Pi^d| t.
\end{equation*}

\medskip

Step 2.~Let us fix $t$ and divide interval $[0,t]$ into $k$ subintervals: $[0,\frac{t}{k}]$, $[\frac{t}{k},\frac{2t}{k}],\dots$, \dots, $[\frac{(k-1)t}{n},t]$, where 
$k$ is chosen sufficiently large so that
$$
\gamma:=c|\Pi^d| \frac{t}{k}<1.
$$ 
Now, let $s_*>0$ be minimal such that $\langle \cosh (\frac{f}{\sqrt{1-\gamma}\,s_*})-1\rangle=1$ (i.e.\,$\|f\|_{\cosh-1}=\sqrt{1-\gamma}\,s_*$).
 Using the Taylor series expansion for $\cosh-1$, one sees that
$$
\cosh (\frac{f}{\sqrt{1-\gamma}\,s_*})-1 \geq \frac{1}{1-\gamma}\biggl[\cosh (\frac{f}{s_*}))-1\biggr].
$$
So, $\langle \cosh (\frac{f}{s_*}))-1 \rangle \leq 1-\gamma$.
Therefore, 
$$
\langle \cosh (\frac{v(\frac{t}{k})}{s_*})-1\rangle \leq 1,
$$
and so $$\|v(\frac{t}{k})\|_{\cosh-1} \leq c_* \equiv \frac{1}{(1-\gamma)^{\scriptscriptstyle \frac{1}{2}}}\|f\|_{\cosh-1} \equiv \frac{1}{\big(1-c|\Pi^d|\frac{t}{k}\big)^{\scriptscriptstyle \frac{1}{2}}}\|f\|_{\cosh-1}.$$
By the semigroup property, 
$$\|v(t)\|_{\cosh-1} \leq  \bigg(1-c|\Pi^d|\frac{t}{k}\bigg)^{-\frac{k}{2}}\|f\|_{\cosh-1}.$$
Taking $k \rightarrow \infty$, we obtain contractivity bound \eqref{v_bd}.
\hfill \qed

\bigskip

\section{Proof of Corollary \ref{cor2}} Proof of existence. Step 1. The argument in Step 1 of the proof of Theorem \ref{thm2} yields
\begin{equation}
\label{unif_T}
\|T^\ast(b)\|_{2 \rightarrow 2}, \quad \sup_{\varepsilon>0}\|T^\ast(b_\varepsilon)\|_{2 \rightarrow 2} \leq C.
\end{equation}
see Remark \ref{dual_rem} (to prove a priori estimates for $b_\varepsilon$ we invoke \eqref{approx_dom}, \eqref{approx_dom2}). Fix some $g \in C^\infty$. For every $\varphi \in C^\infty$,
\begin{align*}
\langle \big(T^\ast(b)-T^\ast(b_\varepsilon)\big)g,\varphi \rangle = \langle (b-b_\varepsilon)(\lambda+A)^{-\frac{1}{2}}g,\nabla (\lambda+A)^{-\frac{1}{2}}\varphi\rangle \rightarrow 0, \quad \varepsilon \downarrow 0.
\end{align*}
Taking into account uniform bounds \eqref{unif_T}, we obtain that $$T^\ast(b_\varepsilon)g \rightarrow T^\ast(b)g \quad \text{ weakly in $L^2$.}
$$
 Fix some $\varepsilon_n \downarrow 0$ and set $b_n:=b_{\varepsilon_n}$. Using the fact that $b_n \mapsto T^\ast(b_n)g$ is linear, we can now apply Mazur's Lemma, which ensures that there exists a sequence of finite convex combinations $\tilde{b}_n=\sum_{k \geq n} c_{k,n}b_k$ such that
 $$
 T^\ast(\tilde{b}_n)g \rightarrow T^\ast(b)g \quad \text{ in }L^2.
 $$
The choice of $\{\tilde{b}_n\}$ depends on $g$.
Using a diagonal argument, we can refine $\{\tilde{b}_n\}$ further, to a sequence $\{\hat{b}_n \}$ of convex combinations $\{b_n\}$ to have the previous convergence hold for all $g$ in a fixed countable subset of $C^\infty$ that is dense in $L^2$. In view of \eqref{unif_T}, this implies that 
\begin{equation}
\label{strong_conv}
 T^\ast(\hat{b}_n) \rightarrow T^\ast(b) \quad \text{ strongly in } L^2.
\end{equation}

Step 2.
The solution $u_n$ to the approximating equation $$
(\lambda + A + \hat{b}_n \cdot \nabla)u_n =f_n, 
$$
where $f_n$ is a mollification of $f$, say, $f_n:=e^{-n^{-1}(-\Delta)^{\frac{\alpha_1}{2}}}f$,
satisfies
\begin{equation}
\label{weak_sol_approx}
\lambda \langle u_n,\varphi \rangle + \langle A^{\frac{1}{2}}u_n,A^{\frac{1}{2}}\varphi\rangle + \langle T(\hat{b}_n)(\lambda+A)^{\frac{1}{2}}u_n,(\lambda+A)^{\frac{1}{2}}\varphi\rangle=\langle f_n,\varphi\rangle
\end{equation}
for all $\varphi \in L^2$ such that $A^{\frac{1}{2}}\varphi \in L^2$. By the a priori estimates \eqref{disp1}, \eqref{disp2}, i.e.
\begin{equation*}
\|u_n\|_\infty \leq \frac{1}{\lambda}\|f\|_\infty.
\end{equation*}
\begin{equation*}
\langle |A^{\frac{1}{2}}u_n|^2\rangle \leq \frac{C}{\lambda^2}\|f\|_\infty^2,
\end{equation*}
we can find a subsequence of $\{u_n\}$ (without loss of generality still denoted by $\{u_n\}$) such that
$$
u_n \rightarrow u \text{ weakly in } L^2 \text{ for some } u, 
$$
$$
A^{\frac{1}{2}}u_n \rightarrow A^{\frac{1}{2}}u \quad \text{weakly in }L^2.
$$
Using \eqref{strong_conv}, we can pass to the limit $n \rightarrow \infty$ in \eqref{weak_sol_approx} to obtain, 
$$
\lambda \langle u,\varphi \rangle + \langle A^{\frac{1}{2}}u,A^{\frac{1}{2}}\varphi\rangle + \langle T(b)(\lambda+A)^{\frac{1}{2}}u,(\lambda+A)^{\frac{1}{2}}\varphi\rangle=\langle f,\varphi\rangle.
$$
Since $u_n$ are uniformly bounded, we have $u \in L^\infty$. Thus, $u$ is a weak solution.

\medskip

Proof of uniqueness. Suppose that there exist two weak solutions: $u_1$ and $u_2$. Then $v=u_1-u_2$ satisfies
\begin{equation}
\label{uniq_id}
\lambda \langle v,\varphi \rangle + \langle A^{\frac{1}{2}}v,A^{\frac{1}{2}}\varphi\rangle + \langle T(\lambda+A)^{\frac{1}{2}}v,(\lambda+A)^{\frac{1}{2}}\varphi\rangle=0.
\end{equation}
We can take $\varphi:=e^{v}$. Indeed, since $v \in L^\infty$, $A^{\frac{1}{2}}v \in L^2$, by the calculation in the proof of Theorem \ref{thm2}, we have $A^{\frac{1}{2}}e^v \in L^2$, so we have an admissible test function:
$$
\lambda \langle v,e^v\rangle + \langle A^{\frac{1}{2}}v,A^{\frac{1}{2}}e^v\rangle + \langle T(\lambda+A)^{\frac{1}{2}}v,(\lambda+A)^{\frac{1}{2}}e^v\rangle=0.
$$
Now, we apply  Theorem \ref{thm2}, i.e.\,Hardy's inequality in the a posteriori form:
$$
\lambda\langle v,e^v\rangle \leq c \langle e^v\rangle.
$$
Replacing $v$ by $-v$ and adding up the resulting inequalities and dividing by $2$ yields
$$
\lambda\langle v \sinh v\rangle \leq c \langle \cosh v\rangle.
$$
We now apply $v \sinh v \geq \cosh v-1$:
$$
(\lambda-c)\langle (\cosh v -1)\rangle \leq c |\Pi^d|.
$$
Note that identity \eqref{uniq_id} also holds for $\frac{v}{s}$ for any $s>0$, and therefore so does the last inequality:
$$
(\lambda-c)\langle (\cosh \frac{v}{s} -1)\rangle \leq c |\Pi^d|.
$$
Selecting $s$ sufficiently small, we can make the left-hand side arbitrarily large thus arriving to a contradiction, unless  $v \equiv 0$. \hfill \qed

\end{document}